\documentclass[11pt,reqno]{amsart}
\usepackage{amsmath,amsxtra,latexsym,amsthm,amssymb,amscd,pb-diagram}
\usepackage{mathrsfs,mathabx,amsfonts}
\usepackage[colorlinks=true,linkcolor=red,citecolor=blue]{hyperref}
\usepackage[margin=1in]{geometry}
\usepackage{bm}
\usepackage{color}
\usepackage{makecell,multirow,diagbox}
\usepackage{mathtools}
\usepackage[displaymath,mathlines]{lineno}
\usepackage{verbatim} 
\usepackage{graphicx}
\usepackage{epsfig}
\usepackage{array}
\usepackage{enumitem}

\newtheorem{theorem}{Theorem}
\newtheorem{corollary}[theorem]{Corollary}
\newtheorem{lemma}[theorem]{Lemma}

\theoremstyle{definition}
\newtheorem{definition}[theorem]{Definition}
\newtheorem{hypothesis}[theorem]{Hypothesis}
\theoremstyle{remark}
\newtheorem{remark}[theorem]{Remark}
\newtheorem{example}{Example}

\begin{document} 
	\title[Nonexistence for damped waves]{Nonexistence for effectively damped waves with time-dependent mass}
	
	
	\maketitle
	\centerline{\scshape \textbf{Duc An Phan}}
	{\footnotesize
		\centerline{Department of Mathematics, Banking Academy of Vietnam}
		\centerline{12 Chua Boc, Kim Lien, Hanoi, Vietnam}
		\centerline{Email: anpd@hvnh.edu.vn}}
	\medskip
	
	\centerline{\scshape \textbf{The Anh Cung}}
	{\footnotesize
		\centerline{School of Mathematics and Computer Science, Hanoi National University of Education}
		\centerline{136 Xuan Thuy, Cau Giay, Hanoi, Vietnam}
		\centerline{Email: anhctmath@hnue.edu.vn}}
	\medskip
	
	\centerline{\scshape \textbf{Trung Loc Tang}}
	{\footnotesize
		\centerline{Department of Mathematics, Thang Long University}
		\centerline{Nghiem Xuan Yem, Hoang Mai, Hanoi, Vietnam}
		\centerline{Email: loctt@thanglong.edu.vn}}
	
	\begin{abstract}
		In this paper, we study the semilinear wave equations
		$$
		u_{tt}-\Delta u+b(t)u_t+m^2(t)u=|u|^p, \quad t \geq 0, \quad x\in\mathbb{R}^n
		$$
		with effective time-dependent damping and a time-dependent mass dominated by the damping. D'Abbicco, Girardi and Reissig established global small-data existence in supercritical ranges and identified the scale
		$$
		p_{\beta,\eta}(n)=1+\frac{2\eta}{n+2\eta\beta}
		$$
		for initial data in $(L^\eta(\mathbb R^n)\cap H^1(\mathbb R^n))\times(L^\eta(\mathbb R^n)\cap L^2(\mathbb R^n))$ with $1\leq\eta<2$, where $\beta$ is the lower mass index associated with the damping-mass pair. To support the expected sharpness of this scale, they also established an analogous subcritical nonexistence result for the corresponding diffusion equation with nonnegative initial data in $L^\eta(\mathbb{R}^n)$, leaving the wave-equation counterpart with effective damping and time-dependent mass open. We address this problem for $\eta=1$ under an intrinsic accumulated-mass balance and a Liouville nonoscillation condition. By constructing a positive slow adjoint mode, we prove nonexistence of global weak solutions for
		$$
		1<p<p_{\beta,1}(n)=1+\frac{2}{n+2\beta},
		$$
		and also treat the critical case $p=p_{\beta,1}(n)$ under an Osgood divergence condition. Conditional lifespan upper bounds and explicit admissible coefficient families are also given.
	\end{abstract}
	
	\tableofcontents
	
	\section{Introduction}
	
	In the present paper, we investigate the Cauchy problem for semilinear waves equations with time-dependent damping and mass terms
	\begin{equation}
		\label{eq:main}
		\begin{cases}
			u_{tt}-\Delta u+b(t)u_t+m^2(t)u=|u|^p,
			&t \geq 0, \quad x\in\mathbb{R}^n,\\
			u(0,x)=\varepsilon f(x), \quad u_t(0,x)=\varepsilon g(x),
		\end{cases}
	\end{equation}
	where $n\geq1$, $p>1$ and $\varepsilon>0$. The damping is assumed to be effective, and the positive mass is dominated by the damping. D'Abbicco, Girardi and Reissig studied precisely the semilinear wave problem \eqref{eq:main} with time-dependent damping and mass. They first derived decay estimates for the associated linear wave equations with time-dependent damping and mass, and then used them to prove global small-data energy solutions in supercritical ranges. They introduced
	\begin{equation}
		\label{eq:beta}
		\beta:=\liminf_{t\to\infty}B(t)m^2(t),
		\quad
		B(t):=\int_0^t\frac{\mathrm{d}\tau}{b(\tau)},
	\end{equation}
	and, for data in the energy space, obtained global existence above
	$$
	\overline p_{\beta}(n):=1+\frac{4}{n+4\beta}
	\quad(\beta<\infty),
	$$
	subject to the dimensional and Sobolev restrictions stated in \cite[Theorem 2]{DGR2019}. With the additional integrability
	$$
	(f,g)\in\mathcal A_\eta
	:=(L^\eta(\mathbb R^n)\cap H^1(\mathbb R^n))\times(L^\eta(\mathbb R^n)\cap L^2(\mathbb R^n)),
	\quad 1\leq\eta<2,
	$$
	they obtained the scale
	$$
	p_{\beta,\eta}(n)=1+\frac{2\eta}{n+2\eta\beta},
	$$
	and in particular
	\begin{equation}
		\label{eq:pc}
		p_{\beta,1}(n):=1+\frac{2}{n+2\beta}.
	\end{equation}
	Their global-existence theorem assumes $p\geq2/\eta$, $p>p_{\beta,\eta}(n)$ and, when $n\geq3$, $p\leq1+2/(n-2)$; see \cite[Theorem 3 and Remark 2.7]{DGR2019}.
	
	To investigate the expected sharpness of these existence results for effectively damped wave equations with time-dependent mass, the authors also considered the corresponding nonlinear diffusion problem.
	\begin{equation}
		\label{eq:diffusive-intro}
		\begin{cases}
			b(t)v_t-\Delta v+m^2(t)v=v^p,
			&t>0,\ x\in\mathbb{R}^n,\\
			v(0,x)=\phi(x)\geq0.
		\end{cases}
	\end{equation}
	This diffusive equation was not the main equation studied in their paper; rather, it served as a comparison model suggested by the diffusion behavior of the linear wave equations with time-dependent damping and mass. In \cite[Section 5]{DGR2019}, they sketched that, for initial data $\phi\in L^\eta(\mathbb R^n)$, $1\leq\eta\leq2$, global small-data Sobolev solutions to \eqref{eq:diffusive-intro} exist for $p>p_{\beta,\eta}(n)$, whereas global Sobolev solutions do not exist for $p<p_{\beta,\eta}(n)$ under a suitable sign assumption. For $\eta>1$, their test-function argument uses initial data satisfying an additional positive lower bound compatible with membership in $L^\eta$. Thus the diffusive problem supports the same scale $p_{\beta,\eta}(n)$ in the strict supercritical and strict subcritical ranges, while the endpoint $p=p_{\beta,\eta}(n)$ is not treated there.
	
	Motivated by this comparison, D'Abbicco, Girardi, and Reissig explicitly conjectured that, under a suitable sign assumption, global Sobolev solutions to the original Cauchy problem \eqref{eq:main} with initial data in $L^\eta$ should likewise fail to exist when $p<p_{\beta,\eta}(n)$. The nonexistence problem left open in their paper therefore concerns the same wave equations with time-dependent damping and mass, not the diffusive equation.
	
	The purpose of the present paper is to address this open problem for effectively damped wave equations with time-dependent mass in the case $\eta=1$, under additional assumptions that allow us to employ a direct adjoint-multiplier argument. More precisely, we retain the effective-damping and dominated-mass hypotheses of \cite{DGR2019}, impose an intrinsic accumulated-mass balance and a global Liouville nonoscillation condition, and prove nonexistence throughout the strict subcritical range at the scale \eqref{eq:pc}. We also treat the critical exponent $p=p_{\beta,1}(n)$ under the corresponding Osgood divergence condition. Hence, our result answers the open question concerning effectively damped wave equations with time-dependent mass within the globally nonoscillatory subclass of the coefficient regime considered in \cite{DGR2019}. In Section 7 we present explicit coefficient families satisfying simultaneously Hypotheses 1--2 of \cite{DGR2019} and all coefficient assumptions required here. Our proof belongs to the modified test-function approach developed in \cite{DAbbiccoLucente2013,IkedaSobajima2019,ISW2019}, while the diffusion clock $B(t)$ is consistent with the effective-damping theory in \cite{Wirth2007}.
	
	A key step in the proof is to construct a positive slow solution of the adjoint equation
	$$
	\rho''-(b\rho)'+m^2\rho=0.
	$$
	To guarantee its positivity on the whole half-line, we impose the
	Liouville nonoscillation condition
	$$
	m^2(t)\leq\frac{b^2(t)}4+\frac{b'(t)}2,
	\quad t\geq0.
	$$
	Indeed, under the Liouville transformation, this condition prevents zeros of the slow adjoint mode on compact time intervals.
	
	Moreover, the lower mass index $\beta$ alone does not control the
	accumulated mass
	$$
	Q(t,r):=\int_r^t\frac{m^2(\tau)}{b(\tau)},\mathrm{d}\tau.
	$$
	We therefore introduce the intrinsic excess
	$$
	\Omega_{\beta,T_0}(L)
	:=
	\sup_{\substack{t\geq r\geq T_0\
			\log X(t,r)\leq L}}
	\bigl[Q(t,r)-\beta\log X(t,r)\bigr]_+,
	\quad
	X(t,r):=\frac{1+B(t)}{1+B(r)}.
	$$
	This quantity controls the weighted volume arising in the test-function argument and yields the threshold exponent \eqref{eq:pc}, together with the Osgood condition at the critical power.
	
	\textbf{Notations:} We write $f\lesssim g$ when there exists a constant $C>0$ such that $f\le Cg$, and $f \sim g$ when $g\lesssim f\lesssim g$. For a real number $s$, we set $[s]_+:=\max\{s,0\}$. When $T=\infty$, an integral over $[0,T)$ is understood as an
	integral over $[0,\infty)$. All statements concerning inequalities
	between functions are understood to hold almost everywhere whenever
	the functions involved are only measurable.
	
	\textbf{The remainder of the paper is organized as follows.} Section \ref{sec:coefficients} introduces the assumptions on the
	damping and mass coefficients, the intrinsic accumulated-mass
	balance, and the global Liouville nonoscillation condition. It also
	derives preliminary estimates for the diffusion scale and the
	coefficients arising in the adjoint analysis.
	Section \ref{sec:adjoint-mode} constructs a global positive slow
	solution of the adjoint equation by combining a terminal Volterra
	equation on a sufficiently large time interval with a backward
	Liouville nonoscillation argument. Section \ref{sec:weak-cutoffs} defines weak solutions and establishes
	the diffusion-scale cutoff and weighted-volume estimates.
	Section \ref{sec:fundamental-osgood} derives the fundamental
	test-function inequality and the Osgood-type lemma. Section \ref{sec:main-results} proves the nonexistence and lifespan
	results, including the critical case. Section \ref{sec:examples} presents explicit coefficient families satisfying both the assumptions of \cite{DGR2019} and those imposed in this paper. Finally, Section \ref{sec:conclusion} summarizes the main conclusions.
	
	\section{Coefficient assumptions, intrinsic balance, and preliminary estimates} \label{sec:coefficients}
	
	We first recall the standard effective-damping and dominated-mass hypotheses adopted in the framework of \cite{DGR2019}.
	
	\begin{hypothesis}[Effective damping]
		\label{ass:b}
		The function $b\in \mathcal{C}^2([0,\infty))$ is positive and monotone. Moreover,
		\begin{align}
			\label{eq:b-der}
			&|b^{(k)}(t)|\leq Cb(t)(1+t)^{-k},
			\quad k=1,2, \\
			\label{eq:b-eff}
			&tb(t)\longrightarrow\infty,
			\quad
			\frac{1}{b(t)(1+t)^2}\in L^1(0,\infty),
		\end{align}
		and
		\begin{equation}
			\label{eq:b-growth}
			tb'(t)\leq a_0b(t)
			\quad \text{ for some } a_0\in[0,1).
		\end{equation}
	\end{hypothesis}
	
	\begin{hypothesis}[Dominated mass]
		\label{ass:m}
		The function $m\in \mathcal{C}^2([0,\infty))$ is positive and
		$$
		|m^{(k)}(t)|\leq Cm(t)(1+t)^{-k},
		\quad k=1,2, \quad \text{ while } \quad m(t)=o(b(t))
		\quad(t\to\infty).
		$$
	\end{hypothesis}
	
	For $t\geq r$, set
	$$
	B(t,r):=\int_r^t\frac{\mathrm{d}\tau}{b(\tau)},
	\quad
	B(t):=B(t,0),
	\quad
	Q(t,r):=\int_r^t\frac{m^2(\tau)}{b(\tau)}\mathrm{d}\tau, \quad
	X(t,r):=\frac{1+B(t)}{1+B(r)}\geq1.
	$$
	
	The preceding two hypotheses are inherited from the framework of
	\cite{DGR2019}. To control the accumulated influence of the mass term
	and to construct a global positive slow solution of the adjoint
	equation, we impose the following two additional hypotheses.
	
	\begin{hypothesis}[Intrinsic accumulated-mass balance]
		\label{ass:balance}
		Let $\beta\in[0,\infty)$ be the finite lower index defined in \eqref{eq:beta}. There exists $T_0\geq0$ such that
		$$
		\Omega_{\beta,T_0}(L):=
		\sup_{\substack{t\geq r\geq T_0\\
				\log X(t,r)\leq L}}
		\bigl[Q(t,r)-\beta\log X(t,r)\bigr]_+<\infty
		\quad \text{ for every } L\geq0.
		$$
	\end{hypothesis}
	
	\begin{hypothesis}[Global Liouville nonoscillation]
		\label{ass:liouville}
		For every $t\geq0$,
		\begin{equation}
			\label{eq:liouville}
			m^2(t)\leq\frac{b^2(t)}4+\frac{b'(t)}2.
		\end{equation}
	\end{hypothesis}
	
	\begin{remark}
		Hypothesis \ref{ass:liouville} is equivalent to the nonnegativity of
		$$
		W(t):=\frac{b^2(t)}4+\frac{b'(t)}2-m^2(t).
		$$
		It is automatically satisfied for all sufficiently large $t$. Indeed,
		$$
		\frac{m^2(t)}{b^2(t)}\longrightarrow0,
		\quad
		\left|\frac{b'(t)}{b^2(t)}\right|
		\leq\frac{C}{(1+t)b(t)}\longrightarrow0,
		$$
		so eventually
		$$
		\frac{m^2(t)}{b^2(t)}\leq\frac18,
		\quad
		\frac{b'(t)}{2b^2(t)}\geq-\frac18,
		$$
		which gives $W(t)\geq0$. A convenient separated sufficient condition is the existence of $\delta\in[0,1/4)$ such that
		$$
		b'(t)\geq-2\delta b^2(t),
		\quad
		m^2(t)\leq\left(\frac14-\delta\right)b^2(t)
		\quad(t\geq0).
		$$
		Indeed, these inequalities imply
		$$
		\frac{b^2(t)}4+\frac{b'(t)}2
		\geq\left(\frac14-\delta\right)b^2(t)\geq m^2(t).
		$$
		In particular, if $b$ is nondecreasing, then $m(t)\leq b(t)/2$ is sufficient.
	\end{remark}
	
	\begin{definition}
		\label{def:sublinear}
		The intrinsic excess is called sublinear if
		$$
		\lim_{L\to\infty}\frac{\Omega_{\beta,T_0}(L)}{L}=0.
		$$
	\end{definition}
	
	\begin{remark}
		By definition,
		\begin{equation}
			\label{eq:balance-exp}
			e^{Q(t,r)}
			\leq X(t,r)^\beta
			\exp\bigl(\Omega_{\beta,T_0}(\log X(t,r))\bigr)
			\quad(t\geq r\geq T_0).
		\end{equation}
		Moreover, $\Omega_{\beta,T_0}$ is nonnegative and nondecreasing. Since $X(t,r)\geq1$, the condition $\log X(t,r)\leq0$ forces $t=r$; hence
		$$
		\Omega_{\beta,T_0}(0)=0.
		$$
	\end{remark}
	
	\begin{remark}[The lower index does not imply sublinear intrinsic excess]
		\label{rem:oscillatory-excess}
		Let $b(t)\equiv1$, let $\beta_0>0$ and $\delta>0$, and define
		$$
		m^2(t):=\frac{\beta_0+\delta\bigl(1+\sin(\log(1+t))\bigr)}{1+t}.
		$$
		Then Hypotheses \ref{ass:b}--\ref{ass:m} hold. Indeed, set
		$$
		a(t):=\beta_0+\delta\bigl(1+\sin(\log(1+t))\bigr).
		$$
		Since
		$$
		\beta_0\leq a(t)\leq\beta_0+2\delta,
		$$
		we have $m(t)=(1+t)^{-1/2}a(t)^{1/2}>0$. Moreover,
		$$
		a'(t)=\frac{\delta\cos(\log(1+t))}{1+t}
		$$
		and
		$$
		a''(t)=
		-\frac{\delta\bigl(\sin(\log(1+t))+\cos(\log(1+t))\bigr)}{(1+t)^2}.
		$$
		Hence
		$$
		|a'(t)|\leq\frac{\delta}{1+t},
		\quad
		|a''(t)|\leq\frac{2\delta}{(1+t)^2}.
		$$
		Differentiating $m=(1+t)^{-1/2}a^{1/2}$ gives
		$$
		\frac{m'(t)}{m(t)}
		=-\frac{1}{2(1+t)}+\frac{a'(t)}{2a(t)},
		$$
		and therefore
		$$
		|m'(t)|\leq C m(t)(1+t)^{-1}.
		$$
		A second differentiation yields
		$$
		\frac{m''(t)}{m(t)}
		=\left(\frac{m'(t)}{m(t)}\right)^2
		+\frac{1}{2(1+t)^2}
		+\frac{a''(t)}{2a(t)}
		-\frac{(a'(t))^2}{2a^2(t)},
		$$
		so
		$$
		|m''(t)|\leq C m(t)(1+t)^{-2}.
		$$
		Finally,
		$$
		m(t)\sim(1+t)^{-1/2}=o(1)=o(b(t)).
		$$
		Thus the derivative and domination conditions in Hypothesis \ref{ass:m} are satisfied, while Hypothesis \ref{ass:b} is immediate for $b\equiv1$.
		
		In this example,
		$$
		B(t)=t.
		$$
		Since
		$$
		B(t)m^2(t)
		=\frac{t}{1+t}
		\left[\beta_0+\delta\bigl(1+\sin(\log(1+t))\bigr)\right],
		$$
		we obtain
		$$
		\liminf_{t\to\infty}B(t)m^2(t)=\beta_0.
		$$
		Furthermore, for $t\geq r\geq0$,
		$$
		Q(t,r)=\int_r^t
		\frac{\beta_0+\delta+\delta\sin(\log(1+\tau))}{1+\tau}
		\,\mathrm d\tau=(\beta_0+\delta)\log\frac{1+t}{1+r}
		+\delta\bigl(\cos(\log(1+r))-\cos(\log(1+t))\bigr).
		$$
		Because
		$$
		X(t,r)=\frac{1+t}{1+r},
		$$
		it follows that
		$$
		Q(t,r)-\beta_0\log X(t,r)
		=\delta\log X(t,r)+\delta\bigl(\cos(\log(1+r))-\cos(\log(1+t))\bigr).
		$$
		Therefore, whenever $\log X(t,r)\leq L$,
		$$
		Q(t,r)-\beta_0\log X(t,r)\leq\delta L+2\delta,
		$$
		which gives
		$$
		\Omega_{\beta_0,0}(L)\leq\delta L+2\delta.
		$$
		On the other hand, choosing $r=0$ and $t=e^L-1$, we have $\log X(t,0)=L$ and
		$$
		Q(t,0)-\beta_0\log X(t,0)
		=\delta L+\delta(1-\cos L)\geq\delta L.
		$$
		Consequently,
		$$
		\delta L\leq\Omega_{\beta_0,0}(L)\leq\delta L+2\delta,
		$$
		and hence
		$$
		\lim_{L\to\infty}
		\frac{\Omega_{\beta_0,0}(L)}{L}=\delta>0.
		$$
		Thus the intrinsic excess is finite on every bounded interval, so Hypothesis \ref{ass:balance} holds, but it is not sublinear. In particular, the lower index $\beta$ alone does not control the accumulated mass at a subpower level.
		
		If, in addition,
		$$
		\beta_0+2\delta\leq\frac14,
		$$
		then
		$$
		m^2(t)\leq\frac14
		=\frac{b^2(t)}4+\frac{b'(t)}2,
		$$
		so Hypothesis \ref{ass:liouville} also holds. Therefore sublinear intrinsic excess is not a consequence even of Hypotheses \ref{ass:b}--\ref{ass:m}, and \ref{ass:liouville} taken together.
	\end{remark}
	
	\begin{lemma}[Diffusion scale]
		\label{lem:B}
		Under Hypothesis \ref{ass:b},
		\begin{equation}
			\label{eq:Bscale}
			1+B(t)\sim1+\frac{t}{b(t)}.
		\end{equation}
		Moreover,
		\begin{equation}
			\label{eq:B-infinity}
			B(t)\longrightarrow\infty
			\quad(t\to\infty).
		\end{equation}
		In particular,
		\begin{equation}
			\label{eq:Bdoubling}
			\sup_{t\geq0}\frac{1+B(2t+1)}{1+B(t)}<\infty,
		\end{equation}
		and
		\begin{equation}
			\label{eq:bB}
			\sup_{t\geq0}\frac{1}{b^2(t)(1+B(t))}<\infty.
		\end{equation}
	\end{lemma}
	
	\begin{proof}
		We first prove the assertion for large $t$. From \eqref{eq:b-der},
		$$
		\left|\log\frac{b(\tau)}{b(t)}\right|
		\leq C\int_\tau^t\frac{\mathrm{d}s}{1+s}
		\leq C\log2,
		\quad \frac{t}{2}\leq\tau\leq t.
		$$
		Hence $b(\tau)\sim b(t)$ uniformly on $[t/2,t]$, and
		$$
		B(t)\geq\int_{t/2}^t\frac{\mathrm{d}\tau}{b(\tau)}
		\geq c\frac{t}{b(t)}.
		$$
		If $b$ is nonincreasing, then $b(\tau)\geq b(t)$ for $0\leq\tau\leq t$, so
		$$
		B(t)\leq\frac{t}{b(t)}.
		$$
		If $b$ is nondecreasing, then for $1\leq\tau\leq t$, \eqref{eq:b-growth} yields
		$$
		\log\frac{b(t)}{b(\tau)}
		=\int_\tau^t\frac{b'(s)}{b(s)}\,\mathrm{d}s
		\leq a_0\int_\tau^t\frac{\mathrm{d}s}{s}
		=a_0\log\frac{t}{\tau},
		$$
		whence
		$$
		b(\tau)\geq b(t)\left(\frac{\tau}{t}\right)^{a_0}.
		$$
		Therefore,
		$$
		\int_1^t\frac{\mathrm{d}\tau}{b(\tau)}
		\leq\frac{t^{a_0}}{b(t)}\int_1^t\tau^{-a_0}\,\mathrm{d}\tau
		\leq\frac{t}{(1-a_0)b(t)}.
		$$
		The integral over $[0,1]$ is constant, and continuity handles bounded $t$. This proves \eqref{eq:Bscale}. If $b$ is nonincreasing, then $b(t)\leq b(0)$ and $t/b(t)\geq t/b(0)\to\infty$. If $b$ is nondecreasing, the preceding integration with $\tau=1$ gives $b(t)\leq C t^{a_0}$ for $t\geq1$, and hence $t/b(t)\geq C^{-1}t^{1-a_0}\to\infty$. Relation \eqref{eq:Bscale} therefore proves \eqref{eq:B-infinity}.
		
		Next, \eqref{eq:b-der} gives
		$$
		\left|\log\frac{b(2t+1)}{b(t)}\right|
		\leq C\int_t^{2t+1}\frac{\mathrm{d}s}{1+s}
		\leq C\log2,
		$$
		so $b(2t+1)\sim b(t)$. Combining this with \eqref{eq:Bscale} gives \eqref{eq:Bdoubling}. Finally,
		$$
		b^2(t)(1+B(t))\sim b^2(t)+tb(t).
		$$
		The right-hand side tends to infinity because $tb(t)\to\infty$ and is positive on every compact interval. Hence its reciprocal is bounded, proving \eqref{eq:bB}.
	\end{proof}
	
	\begin{lemma}[Asymptotic bounds for the adjoint coefficients]
		\label{lem:AEH}
		Under Hypotheses \ref{ass:b}--\ref{ass:balance}, define
		\begin{equation}
			\label{eq:AEH}
			q(t):=\frac{m^2(t)}{b(t)},
			\quad
			A(t):=q(t)-\frac{b'(t)}{b(t)},
			\quad
			E(t):=A'(t)+A^2(t),
			\quad
			H(t):=b(t)-2A(t).
		\end{equation}
		Then
		$$
		q(t)\lesssim(1+t)^{-1},
		\quad
		|A(t)|\lesssim(1+t)^{-1},
		\quad
		|E(t)|\lesssim(1+t)^{-2}.
		$$
		Moreover,
		$$
		\frac{A(t)}{b(t)}\longrightarrow0,
		\quad
		H(t)\sim b(t)
		\quad \text{ for all sufficiently large } t,
		$$
		and
		$$
		\frac{|H'(t)|}{H^2(t)}\longrightarrow0.
		$$
		Moreover,
		\begin{equation}
			\label{eq:E-tail}
			\int_T^\infty\frac{|E(t)|}{H(t)}\mathrm{d}t<\infty
		\end{equation}
		for every sufficiently large $T$.
	\end{lemma}
	
	\begin{proof}
		Since $q=m^2/b>0$, the derivative assumptions give
		$$
		\frac{q'(t)}{q(t)}
		=2\frac{m'(t)}{m(t)}-\frac{b'(t)}{b(t)},
		\quad
		\left|\frac{q'(t)}{q(t)}\right|\leq\frac{C}{1+t}.
		$$
		For $t\leq\tau\leq2t+1$ and large $t$, integration yields
		$$
		\log\frac{q(\tau)}{q(t)}
		\geq-C\log\frac{1+\tau}{1+t}
		\geq-C\log2,
		$$
		so $q(\tau)\geq c q(t)$ with $c>0$ independent of $t$. By \eqref{eq:Bdoubling}, there exists $L_0>0$ such that
		$$
		0\leq\log X(2t+1,t)\leq L_0
		$$
		for all sufficiently large $t$. Taking $t\geq T_0$, Hypothesis \ref{ass:balance} gives
		$$
		Q(2t+1,t)
		\leq\beta\log X(2t+1,t)
		+\Omega_{\beta,T_0}(\log X(2t+1,t))
		\leq \beta L_0+\Omega_{\beta,T_0}(L_0).
		$$
		Consequently,
		$$
		(1+t)q(t)
		\leq C\int_t^{2t+1}q(\tau)\,\mathrm{d}\tau
		=C Q(2t+1,t)\leq C,
		$$
		which proves $q(t)\lesssim(1+t)^{-1}$.
		
		The logarithmic derivative estimate then gives
		$$
		|q'(t)|\leq\frac{Cq(t)}{1+t}\leq\frac{C}{(1+t)^2}.
		$$
		Moreover,
		$$
		\left(\frac{b'}{b}\right)'
		=\frac{b''}{b}-\left(\frac{b'}{b}\right)^2,
		$$
		and \eqref{eq:b-der} implies
		$$
		\left|\left(\frac{b'}{b}\right)'\right|\leq\frac{C}{(1+t)^2}.
		$$
		Thus
		$$
		|A(t)|\leq\frac{C}{1+t},
		\quad
		|A'(t)|\leq\frac{C}{(1+t)^2},
		$$
		and hence
		$$
		|E(t)|=|A'(t)+A^2(t)|\leq\frac{C}{(1+t)^2}.
		$$
		Since $tb(t)\to\infty$,
		$$
		\frac{|A(t)|}{b(t)}\leq\frac{C}{(1+t)b(t)}\longrightarrow0.
		$$
		Therefore $H(t)=b(t)-2A(t)$ is positive for all sufficiently large $t$ and $H(t)/b(t)\to1$.
		
		Because $H'=b'-2A'$, for sufficiently large $t$,
		$$
		\frac{|H'(t)|}{H^2(t)}
		\leq C\left(
		\frac{|b'(t)|}{b^2(t)}+\frac{|A'(t)|}{b^2(t)}
		\right)
		\leq C\left(
		\frac{1}{(1+t)b(t)}+\frac{1}{(1+t)^2b^2(t)}
		\right)\longrightarrow0.
		$$
		Finally, $H(t)\sim b(t)$ and the estimate for $E$ imply
		$$
		\frac{|E(t)|}{H(t)}\leq\frac{C}{b(t)(1+t)^2}
		$$
		for all sufficiently large $t$. The last function is integrable by \eqref{eq:b-eff}, which proves \eqref{eq:E-tail}.
	\end{proof}
	
	\section{Construction of a global positive slow adjoint mode} \label{sec:adjoint-mode}
	
	\begin{lemma}[Backward kernel estimate]
		\label{lem:kernel}
		Let $T\geq0$, let $H\in \mathcal{C}^1([T,\infty))$ be positive, and assume
		$$
		\kappa_T:=\sup_{t\geq T}\frac{|H'(t)|}{H^2(t)}<1.
		$$
		Then, for $\sigma\geq t\geq T$,
		\begin{equation}
			\label{eq:kernel-est}
			K_T(t,\sigma):=
			\int_t^\sigma
			\exp\left(-\int_\tau^\sigma H(r)\mathrm{d}r\right)\mathrm{d}\tau
			\leq\frac{1}{(1-\kappa_T)H(\sigma)},
		\end{equation}
		and
		\begin{equation}
			\label{eq:Htransport}
			H(\sigma)
			\exp\left(-\int_t^\sigma H(r)\mathrm{d}r\right)
			\leq H(t).
		\end{equation}
	\end{lemma}
	
	\begin{proof}
		Fix $\sigma\geq t$ and define
		$$
		F_\sigma(\tau):=\frac{1}{H(\tau)}
		\exp\left(-\int_\tau^\sigma H(r)\,\mathrm{d}r\right),
		\quad t\leq\tau\leq\sigma.
		$$
		A direct differentiation gives
		$$
		F_\sigma'(\tau)
		=\exp\left(-\int_\tau^\sigma H(r)\,\mathrm{d}r\right)
		\left(1-\frac{H'(\tau)}{H^2(\tau)}\right).
		$$
		Since $|H'|/H^2\leq\kappa_T<1$,
		$$
		F_\sigma'(\tau)
		\geq(1-\kappa_T)
		\exp\left(-\int_\tau^\sigma H(r)\,\mathrm{d}r\right).
		$$
		Integration from $t$ to $\sigma$ yields
		$$
		(1-\kappa_T)K_T(t,\sigma)
		\leq F_\sigma(\sigma)-F_\sigma(t)
		\leq\frac{1}{H(\sigma)},
		$$
		which proves \eqref{eq:kernel-est}.
		
		For \eqref{eq:Htransport},
		$$
		\log\frac{H(\sigma)}{H(t)}
		=\int_t^\sigma\frac{H'(r)}{H(r)}\,\mathrm{d}r
		\leq\int_t^\sigma\frac{|H'(r)|}{H^2(r)}H(r)\,\mathrm{d}r
		\leq\kappa_T\int_t^\sigma H(r)\,\mathrm{d}r.
		$$
		Exponentiating and multiplying by $\exp(-\int_t^\sigma H)$ gives
		$$
		H(\sigma)\exp\left(-\int_t^\sigma H(r)\,\mathrm{d}r\right)
		\leq H(t)
		\exp\left(-(1-\kappa_T)\int_t^\sigma H(r)\,\mathrm{d}r\right)
		\leq H(t).
		$$
	\end{proof}
	
	\begin{theorem}[Global positive slow adjoint mode]
		\label{thm:rho}
		Assume Hypotheses \ref{ass:b}--\ref{ass:liouville}. Then the adjoint equation
		\begin{equation}
			\label{eq:adjoint}
			\rho''-(b\rho)'+m^2\rho=0
		\end{equation}
		admits a solution $\rho\in \mathcal{C}^2([0,\infty))$ such that
		\begin{align}
			\label{eq:rho-sign}
			&\rho(t)>0,
			\quad
			b(t)\rho(t)-\rho'(t)>0
			\quad(t\geq0), \\
			\label{eq:rho-der}
			&|\rho'(t)|\leq Cb(t)\rho(t)
			\quad(t\geq0),
		\end{align}
		and
		\begin{equation}
			\label{eq:rho-growth}
			b(t)\rho(t)
			\leq C(1+B(t))^\beta
			\exp\bigl(\Omega_{\beta,T_0}(\log(1+B(t)))\bigr)
			\quad(t\geq0).
		\end{equation}
		More precisely, there exists $T_*\geq T_0$ such that
		$$
		\rho(t)\sim\frac{1}{b(t)}e^{Q(t,T_*)},
		\quad t\geq T_*,
		$$
		and
		\begin{equation}
			\label{eq:rho-small-der}
			\left|\frac{\rho'(t)}{b(t)\rho(t)}\right|\leq\frac14,
			\quad t\geq T_*.
		\end{equation}
	\end{theorem}
	
	\begin{proof}
		Let $A,E,H$ be defined by \eqref{eq:AEH}. Lemma \ref{lem:AEH} gives
		$$
		\frac{H(t)}{b(t)}\longrightarrow1,
		\quad
		\frac{|H'(t)|}{H^2(t)}\longrightarrow0,
		\quad
		\int_T^\infty\frac{|E(t)|}{H(t)}\,\mathrm{d}t<\infty
		$$
		for all sufficiently large $T$. Choose $T_*\geq T_0$ so large that
		$$
		H(t)>0,
		\quad
		\kappa_*:=\sup_{t\geq T_*}\frac{|H'(t)|}{H^2(t)}\leq\frac12, \quad
		\nu_*:=\frac{1}{1-\kappa_*}
		\int_{T_*}^\infty\frac{|E(\sigma)|}{H(\sigma)}\,\mathrm{d}\sigma
		\leq\frac14,
		\quad
		\sup_{t\geq T_*}\frac{|A(t)|}{b(t)}\leq\frac18,
		$$
		and
		$$
		\sup_{t\geq T_*}
		\frac{H(t)}{b(t)}
		\int_t^\infty\frac{|E(\sigma)|}{H(\sigma)}\,\mathrm{d}\sigma
		\leq\frac{1}{16}.
		$$
		The last choice is possible because $H/b\to1$ and the tail integral tends to zero.
		
		Set
		$$
		w_*(t):=\frac{e^{Q(t,T_*)}}{b(t)},
		\quad t\geq T_*.
		$$
		Since $Q_t(t,T_*)=q(t)$,
		$$
		\frac{w_*'(t)}{w_*(t)}
		=q(t)-\frac{b'(t)}{b(t)}=A(t).
		$$
		Therefore
		$$
		w_*''=(A'+A^2)w_*=Ew_*.
		$$
		Moreover, $bw_*=e^{Q(t,T_*)}$, so
		$$
		(bw_*)'=q e^{Q(t,T_*)}=m^2w_*.
		$$
		Consequently,
		$$
		w_*''-(bw_*)'+m^2w_*=Ew_*.
		$$
		Writing $\rho=w_*v$ and expanding every derivative gives
		\begin{align*}
			\rho''-(b\rho)'+m^2\rho&=w_*v''+(2w_*'-bw_*)v'
			+\bigl(w_*''-(bw_*)'+m^2w_*\bigr)v\\
			&=w_*\bigl(v''-(b-2A)v'+Ev\bigr)
			=w_*\bigl(v''-Hv'+Ev\bigr).
		\end{align*}
		Thus it suffices to solve
		\begin{equation}
			\label{eq:vode}
			v''-Hv'+Ev=0
			\quad \text{ on } [T_*,\infty)
		\end{equation}
		with
		$$
		\lim_{t\to\infty}v(t)=1,
		\quad
		\lim_{t\to\infty}v'(t)=0.
		$$
		
		Define, for bounded continuous $h$,
		$$
		(\mathcal Th)(t):=-\int_t^\infty K_{T_*}(t,\sigma)E(\sigma)h(\sigma)\,\mathrm{d}\sigma.
		$$
		Lemma \ref{lem:kernel} gives
		$$
		\|\mathcal Th\|_\infty
		\leq\frac{1}{1-\kappa_*}
		\left(\int_{T_*}^\infty\frac{|E(\sigma)|}{H(\sigma)}\,\mathrm{d}\sigma\right)
		\|h\|_\infty
		=\nu_*\|h\|_\infty.
		$$
		For $t$ in a compact interval $J\subset[T_*,\infty)$, the kernel is continuous in $(t,\sigma)$ for $\sigma\geq t$ and satisfies
		$$
		|K_{T_*}(t,\sigma)E(\sigma)h(\sigma)|
		\leq
		\frac{\|h\|_\infty}{1-\kappa_*}
		\frac{|E(\sigma)|}{H(\sigma)},
		$$
		where the right-hand side is integrable. After extending the integrand by zero to $\sigma<t$, dominated convergence proves continuity of $\mathcal Th$. The same estimate also gives
		$$
		|(\mathcal Th)(t)|
		\leq
		\frac{\|h\|_\infty}{1-\kappa_*}
		\int_t^\infty\frac{|E(\sigma)|}{H(\sigma)}\,\mathrm d\sigma
		\longrightarrow0,
		$$
		so $\mathcal T$ maps $C_b([T_*,\infty))$ into itself. Let $v_0\equiv1$ and define recursively $v_{j+1}=1+\mathcal Tv_j$. Then
		$$
		\|v_{j+1}-v_j\|_\infty
		\leq\nu_*^j\|v_1-v_0\|_\infty,
		$$
		so $(v_j)$ converges uniformly to a bounded continuous function $v$ satisfying
		$$
		v(t)=1-\int_t^\infty K_{T_*}(t,\sigma)E(\sigma)v(\sigma)\,\mathrm{d}\sigma.
		$$
		If $v$ and $\widetilde v$ are two bounded solutions, then
		$$
		\|v-\widetilde v\|_\infty
		\leq\nu_*\|v-\widetilde v\|_\infty,
		$$
		so uniqueness follows from $\nu_*<1$. Furthermore,
		$$
		\|v\|_\infty\leq1+\nu_*\|v\|_\infty,
		\quad
		\|v-1\|_\infty\leq\nu_*\|v\|_\infty,
		$$
		whence
		$$
		\|v\|_\infty\leq\frac{1}{1-\nu_*}\leq\frac43,
		\quad
		\|v-1\|_\infty\leq\frac{\nu_*}{1-\nu_*}\leq\frac13.
		$$
		Therefore
		\begin{equation}
			\label{eq:v-tail-bounds}
			\frac23\leq v(t)\leq\frac43,
			\quad t\geq T_*.
		\end{equation}
		Also,
		$$
		|v(t)-1|
		\leq\frac{\|v\|_\infty}{1-\kappa_*}
		\int_t^\infty\frac{|E(\sigma)|}{H(\sigma)}\,\mathrm{d}\sigma
		\longrightarrow0.
		$$
		
		Let $J\Subset[T_*,\infty)$. By \eqref{eq:Htransport}, for $t\in J$ and $\sigma\geq t$,
		$$
		\exp\left(-\int_t^\sigma H(r)\,\mathrm d r\right)
		\leq
		\frac{\sup_{t\in J}H(t)}{H(\sigma)}.
		$$
		Thus the derivative of the Volterra integrand is dominated, uniformly for $t\in J$, by a constant multiple of $|E(\sigma)|/H(\sigma)$. The Leibniz rule and dominated convergence are therefore applicable. Since
		$$
		\partial_tK_{T_*}(t,\sigma)
		=-\exp\left(-\int_t^\sigma H(r)\,\mathrm{d}r\right),
		\quad K_{T_*}(t,t)=0,
		$$
		we obtain
		\begin{equation}
			\label{eq:vprime-tail}
			v'(t)=\int_t^\infty
			\exp\left(-\int_t^\sigma H(r)\,\mathrm{d}r\right)
			E(\sigma)v(\sigma)\,\mathrm{d}\sigma.
		\end{equation}
		The right-hand side in \eqref{eq:vprime-tail} is continuously differentiable on every compact subinterval: the boundary contribution is $-E(t)v(t)$, while differentiation of the exponential contributes $H(t)v'(t)$. Hence
		$$
		v''(t)=-E(t)v(t)+H(t)v'(t),
		$$
		so $v\in \mathcal{C}^2([T_*,\infty))$ and satisfies \eqref{eq:vode}. By \eqref{eq:Htransport},
		$$
		|v'(t)|
		\leq\|v\|_\infty H(t)
		\int_t^\infty\frac{|E(\sigma)|}{H(\sigma)}\,\mathrm{d}\sigma.
		$$
		Hence
		$$
		\frac{|v'(t)|}{b(t)v(t)}
		\leq2\frac{H(t)}{b(t)}
		\int_t^\infty\frac{|E(\sigma)|}{H(\sigma)}\,\mathrm{d}\sigma
		\leq\frac18,
		\quad t\geq T_*.
		$$
		Together with $|A|/b\leq1/8$, this yields
		$$
		\left|\frac{\rho'(t)}{b(t)\rho(t)}\right|
		=\left|\frac{A(t)}{b(t)}+\frac{v'(t)}{b(t)v(t)}\right|
		\leq\frac14,
		\quad t\geq T_*.
		$$
		The same estimate shows $v'(t)\to0$ because $H(t)\int_t^\infty |E|/H\to0$: if $b$ is nondecreasing, then
		$$
		b(t)\int_t^\infty\frac{\mathrm{d}\sigma}{b(\sigma)(1+\sigma)^2}
		\leq\int_t^\infty\frac{\mathrm{d}\sigma}{(1+\sigma)^2}
		=\frac{1}{1+t},
		$$
		whereas if $b$ is nonincreasing, then $b(t)$ is bounded and the tail of the integrable function $1/[b(\sigma)(1+\sigma)^2]$ tends to zero. Thus the terminal conditions are attained. Relations \eqref{eq:v-tail-bounds} and the definition of $w_*$ prove
		$$
		\rho(t)>0,
		\quad
		\rho(t)\sim\frac{1}{b(t)}e^{Q(t,T_*)},
		\quad t\geq T_*.
		$$
		
		The coefficients of \eqref{eq:adjoint} are continuous. Hence the standard existence-uniqueness theorem for linear ordinary differential equations extends the solution uniquely from $T_*$ to $[0,T_*]$; see \cite[Chapter 1]{CoddingtonLevinson1955}. Define
		$$
		z(t):=\exp\left(-\frac12\int_0^t b(\tau)\,\mathrm{d}\tau\right)\rho(t).
		$$
		A direct computation gives
		$$
		\rho'=e^{\frac12\int_0^t b}\left(z'+\frac{b}{2}z\right), \quad
		\rho''=e^{\frac12\int_0^t b}
		\left(z''+bz'+\frac{b'}2z+\frac{b^2}{4}z\right), \quad (b\rho)'=e^{\frac12\int_0^t b}
		\left(bz'+b'z+\frac{b^2}{2}z\right).
		$$
		Therefore \eqref{eq:adjoint} is equivalent to
		\begin{equation}
			\label{eq:liouville-z}
			z''(t)=W(t)z(t),
			\quad
			W(t):=\frac{b^2(t)}4+\frac{b'(t)}2-m^2(t)\geq0.
		\end{equation}
		At $t=T_*$, $z(T_*)>0$ and
		$$
		z'(T_*)
		=e^{-\frac12\int_0^{T_*}b}
		\left(\rho'(T_*)-\frac12b(T_*)\rho(T_*)\right)
		\leq-\frac14b(T_*)e^{-\frac12\int_0^{T_*}b}\rho(T_*)<0.
		$$
		Let $(t_0,T_*]$ be the maximal interval ending at $T_*$ on which $z>0$. On this interval, \eqref{eq:liouville-z} gives $z''\geq0$, so $z'$ is nondecreasing and
		$$
		z'(t)\leq z'(T_*)<0,
		\quad t_0<t\leq T_*.
		$$
		Thus
		$$
		z(t)=z(T_*)-\int_t^{T_*}z'(\tau)\,\mathrm{d}\tau
		\geq z(T_*)+(T_*-t)|z'(T_*)|>0.
		$$
		If $t_0>0$, continuity would imply $z(t_0)=0$, a contradiction. Hence $t_0=0$; letting $t\downarrow0$ in the last estimate also gives $z(0)>0$. Thus $z>0$ and $z'<0$ on $[0,T_*]$. Consequently,
		$$
		b(t)\rho(t)-\rho'(t)
		=e^{\frac12\int_0^t b}
		\left(\frac12b(t)z(t)-z'(t)\right)>0,
		\quad 0\leq t\leq T_*.
		$$
		For $t\geq T_*$, the same inequality follows from $|\rho'|\leq b\rho/4$. This proves \eqref{eq:rho-sign}.
		
		On the compact interval $[0,T_*]$, the function $t\longmapsto |\rho'(t)|/(b(t)\rho(t))$ is continuous because $b$ and $\rho$ are positive. Hence it is bounded. Together with \eqref{eq:rho-small-der}, this proves \eqref{eq:rho-der}.
		
		Finally, for $t\geq T_*$, the bounds for $v$ give
		$$
		b(t)\rho(t)
		=
		e^{Q(t,T_*)}v(t)
		\leq C e^{Q(t,T_*)}.
		$$
		Since $T_*\geq T_0$, \eqref{eq:balance-exp} yields
		$$
		e^{Q(t,T_*)}
		\leq
		X(t,T_*)^\beta
		\exp\bigl(
		\Omega_{\beta,T_0}(\log X(t,T_*))
		\bigr).
		$$
		Moreover,
		$$
		X(t,T_*)
		=
		\frac{1+B(t)}{1+B(T_*)}
		\leq1+B(t).
		$$
		Since $\beta\geq0$ and $\Omega_{\beta,T_0}$ is nondecreasing, we obtain
		$$
		b(t)\rho(t)
		\leq
		C(1+B(t))^\beta
		\exp\bigl(
		\Omega_{\beta,T_0}(\log(1+B(t)))
		\bigr),
		\quad t\geq T_*.
		$$
		
		For $0\leq t\leq T_*$, since $\beta\geq0$,
		$\Omega_{\beta,T_0}\geq0$, and $B(t)\geq0$, we have
		$$
		(1+B(t))^\beta
		\exp\bigl(
		\Omega_{\beta,T_0}(\log(1+B(t)))
		\bigr)
		\geq1.
		$$
		On the other hand, $b\rho$ is continuous on the compact interval
		$[0,T_*]$, and therefore
		$$
		M_*:=
		\max_{0\leq t\leq T_*}b(t)\rho(t)<\infty.
		$$
		Hence
		$$
		b(t)\rho(t)
		\leq
		M_*
		(1+B(t))^\beta
		\exp\bigl(
		\Omega_{\beta,T_0}(\log(1+B(t)))
		\bigr),
		\quad 0\leq t\leq T_*.
		$$
		Increasing the constant completes the proof of
		\eqref{eq:rho-growth}.
	\end{proof}
	
	\begin{remark}[Necessity of a compact-time nonoscillation restriction]
		\label{rem:obstruction}
		The asymptotic assumptions alone cannot yield the conclusion of Theorem \ref{thm:rho}. To see this, take $b(t)\equiv1$, fix $M>1/2$, and choose
		$$
		L>\frac{\pi}{\sqrt{M^2-1/4}}.
		$$
		Let $\chi\in \mathcal{C}^\infty([0,\infty))$ satisfy $0\leq\chi\leq1$, $\chi=1$ on $[0,L]$, and $\chi=0$ on $[L+1,\infty)$. For a fixed $\beta>0$, define
		$$
		m(t):=\chi(t)M+(1-\chi(t))\sqrt{\frac{\beta}{1+t}}.
		$$
		Then $m$ is positive and smooth, $m(t)=M$ on $[0,L]$, and $m^2(t)=\beta/(1+t)$ for $t\geq L+1$. The derivative estimates in Hypothesis \ref{ass:m} hold on $[L+1,\infty)$ by direct differentiation and on the compact interval $[0,L+1]$ after increasing the constant. Moreover, $m(t)=o(1)=o(b(t))$. With $T_0=L+1$,
		$$
		B(t)=t,
		\quad
		Q(t,r)=\beta\log\frac{1+t}{1+r}
		=\beta\log X(t,r)
		\quad(t\geq r\geq T_0),
		$$
		so $\Omega_{\beta,T_0}\equiv0$.
		
		On $[0,L]$, however, the Liouville transform $\rho=e^{t/2}z$ reduces the adjoint equation to
		$$
		z''+\left(M^2-\frac14\right)z=0.
		$$
		Every nontrivial solution is of the form
		$$
		z(t)=C\cos\left(\sqrt{M^2-\frac14}\,t-\vartheta\right),
		$$
		whose consecutive zeros are separated by $\pi/\sqrt{M^2-1/4}$. Hence every nontrivial solution has a zero in $[0,L]$. This proves that a compact-time nonoscillation restriction is indispensable for the present direct positive-multiplier method.
	\end{remark}
	
	\section{Weak solutions and diffusion-scale cutoffs} \label{sec:weak-cutoffs}
	
	Throughout the rest of the paper, Hypotheses \ref{ass:b}--\ref{ass:liouville} are in force, and $\rho$ denotes
	the positive slow adjoint solution constructed in Theorem \ref{thm:rho}.
	The measure-theoretic steps use H\"older's inequality and the Tonelli, dominated-convergence, and monotone-convergence theorems in their standard forms; see \cite{Folland1999}.
	
	\begin{definition}[Weak solution]
		\label{def:weak}
		Let $T\in(0,\infty]$, let $f,g\in L^1(\mathbb{R}^n)$, and let
		$$
		u\in L^p_{\mathrm{loc}}([0,T)\times\mathbb{R}^n).
		$$
		We call $u$ a weak solution of \eqref{eq:main} on $[0,T)$ if, for every
		$\phi\in \mathcal{C}_c^2([0,T)\times\mathbb{R}^n)$,
		\begin{align}
			\label{eq:weak}
			&\int_0^T\int_{\mathbb{R}^n}|u|^p\phi\,\mathrm{d}x\mathrm{d}t
			+\varepsilon\int_{\mathbb{R}^n}
			\bigl(g\phi(0)+b(0)f\phi(0)-f\phi_t(0)\bigr)\mathrm{d}x
			\notag\\
			&\hspace{2.5cm}
			=\int_0^T\int_{\mathbb{R}^n}u
			\bigl(\phi_{tt}-\Delta\phi-\partial_t(b\phi)+m^2\phi\bigr)
			\mathrm{d}x\mathrm{d}t.
		\end{align}
	\end{definition}
	
	For completeness, define
	$$
	\vartheta(s):=
	\begin{cases}
		0,&s\leq0,\\
		\exp(-1/s),&s>0,
	\end{cases}
	$$
	and
	$$
	\eta(s):=
	\frac{\vartheta(1-s)}
	{\vartheta(1-s)+\vartheta(s-1/2)},
	\quad s\geq0.
	$$
	The denominator is strictly positive for every $s\geq0$. Since
	$$
	\vartheta'(s)=s^{-2}e^{-1/s}>0
	\quad(s>0),
	$$
	while $\vartheta$ and all its derivatives vanish at $s=0$, the quotient is smooth across $s=1/2$ and $s=1$. Moreover, $s\mapsto\vartheta(1-s)$ is nonincreasing and $s\mapsto\vartheta(s-1/2)$ is nondecreasing. Differentiating the quotient therefore gives $\eta'(s)\leq0$. Hence $\eta\in \mathcal{C}^\infty([0,\infty))$, $0\leq\eta\leq1$, $\eta=1$ on $[0,1/2]$, and $\eta=0$ on $[1,\infty)$. For every $\alpha>0$, the function
	$$
	s\longmapsto \vartheta(s)^\alpha
	=
	\begin{cases}
		0,&s\leq0,\\
		\exp(-\alpha/s),&s>0,
	\end{cases}
	$$
	belongs to $\mathcal{C}^\infty(\mathbb R)$ and is flat at $s=0$. Indeed, every
	derivative on $(0,\infty)$ is a finite linear combination of terms of
	the form
	$$
	s^{-N}\exp(-\alpha/s),
	$$
	which tend to zero as $s\downarrow0$ for every $N\geq0$.
	Consequently, for every $\alpha>0$, the function $\eta^\alpha$ is
	smooth on $[0,\infty)$. In particular $\eta(s)^{2p'}\in \mathcal{C}^\infty([0,\infty))$. Set
	$$
	\eta^*(r):=
	\begin{cases}
		0,&0\leq r<1/2,\\
		\eta(r),&r\geq1/2.
	\end{cases}
	$$
	The function $\eta^*$ is only used as a measurable auxiliary weight; it is not used as a test function.
	For $R>1$, define
	\begin{equation}
		\label{eq:cutoff}
		S_R(t,x):=\frac{1+|x|^2+B(t)}{R},
		\quad
		\psi_R:=\eta(S_R)^{2p'},
		\quad
		\psi_R^*:=(\eta^*(S_R))^{2p'},
	\end{equation}
	where $p'=p/(p-1)$. Since $B'(t)=1/b(t)>0$, the function $B$ is continuous and strictly
	increasing. If $T<\infty$, we use these cutoffs only for
	$$
	1<R<1+B(T).
	$$
	Then there exists a unique $t_R\in(0,T)$ such that $B(t_R)=R-1$. When $T=\infty$, relation \eqref{eq:B-infinity} likewise implies that, for every $R>1$, there exists a unique $t_R>0$ satisfying $B(t_R)=R-1$. Since $\eta(s)=0$ for $s\geq1$, in both cases we have
	$$
	\operatorname{supp}\psi_R \subset [0,t_R]\times\overline{B_{\sqrt{R-1}}(0)}.
	$$
	Hence $\psi_R$ is compactly supported in
	$[0,T)\times\mathbb{R}^n$.
	
	\begin{lemma}[Cutoff estimates]
		\label{lem:cutoff}
		On the support of the derivatives of $\psi_R$,
		$$
		|\partial_t\psi_R|
		\leq\frac{C}{Rb(t)}(\psi_R^*)^{1/p}, \quad |\partial_t^2\psi_R|+|\Delta\psi_R|
		\leq\frac{C}{R}(\psi_R^*)^{1/p}.
		$$
	\end{lemma}
	
	\begin{proof}
		Write $F(s):=\eta(s)^{2p'}$. On the transition region $1/2\leq s\leq1$, one has $\eta^*(s)=\eta(s)$ and
		$$
		|F'(s)|\leq C\eta(s)^{2p'-1},
		\quad
		|F''(s)|\leq C\bigl(\eta(s)^{2p'-2}+\eta(s)^{2p'-1}\bigr).
		$$
		Since
		$$
		2p'-2=\frac{2p'}{p},
		\quad
		2p'-1\geq\frac{2p'}{p},
		$$
		and $0\leq\eta\leq1$, it follows that
		$$
		|F'(s)|+|F''(s)|\leq C(\eta^*(s))^{2p'/p}.
		$$
		Moreover,
		$$
		\partial_tS_R=\frac{1}{Rb(t)},
		\quad
		\partial_t^2S_R=-\frac{b'(t)}{Rb^2(t)},
		\quad
		\nabla_xS_R=\frac{2x}{R},
		\quad
		\Delta_xS_R=\frac{2n}{R}.
		$$
		Therefore
		$$
		\partial_t\psi_R=F'(S_R)\partial_tS_R,
		$$
		which proves the first estimate. Also,
		$$
		\partial_t^2\psi_R
		=F''(S_R)(\partial_tS_R)^2+F'(S_R)\partial_t^2S_R.
		$$
		On the support of the derivatives, $S_R\leq1$, so $1+B(t)\leq R$. Hence, by \eqref{eq:bB},
		$$
		\frac{1}{R^2b^2(t)}
		=\frac{1+B(t)}{R^2}\frac{1}{b^2(t)(1+B(t))}
		\leq\frac{C}{R}.
		$$
		Furthermore,
		$$
		\frac{|b'(t)|}{b^2(t)}
		\leq\frac{C}{(1+t)b(t)}
		$$
		is bounded: it is continuous on compact intervals and tends to zero by $tb(t)\to\infty$. Thus $|\partial_t^2\psi_R|\leq C R^{-1}(\psi_R^*)^{1/p}$.
		
		Finally,
		$$
		\Delta\psi_R
		=F''(S_R)|\nabla_xS_R|^2+F'(S_R)\Delta_xS_R.
		$$
		On the derivative support, $|x|^2\leq R$, so
		$$
		|\nabla_xS_R|^2=\frac{4|x|^2}{R^2}\leq\frac4R.
		$$
		Together with $\Delta_xS_R=2n/R$, this proves the second estimate.
	\end{proof}
	
	\begin{lemma}[Weighted volume]
		\label{lem:volume}
		For every $R\geq2$,
		$$
		\int_0^\infty
		\bigl[R-1-B(t)\bigr]_+^{n/2}\rho(t)\mathrm{d}t
		\leq CR^{1+n/2+\beta}
		\exp\bigl(\Omega_{\beta,T_0}(\log R)\bigr).
		$$
	\end{lemma}
	
	\begin{proof}
		Since $\rho$ is continuous on $[0,T_0]$, the contribution of this interval satisfies
		$$
		\int_0^{T_0}
		\bigl[R-1-B(t)\bigr]_+^{n/2}\rho(t)\,\mathrm{d}t
		\leq C R^{n/2}
		\leq C R^{1+n/2+\beta}
		\exp\bigl(\Omega_{\beta,T_0}(\log R)\bigr),
		$$
		because $R\geq2$, $\beta\geq0$, and $\Omega_{\beta,T_0}\geq0$. If $R\leq1+B(T_0)$, the remaining contribution vanishes. Assume
		that $R>1+B(T_0)$. On $[T_0,\infty)$, use \eqref{eq:rho-growth} and the
		change of variable $y=B(t)$, for which $\mathrm{d}t=b(t)\mathrm{d}y$.
		Then
		\begin{align*}
			\int_{T_0}^\infty
			\bigl[R-1-B(t)\bigr]_+^{n/2}\rho(t)\mathrm{d}t &\leq C\int_{B(T_0)}^{R-1}
			(R-1-y)^{n/2}(1+y)^\beta
			\exp\bigl(\Omega_{\beta,T_0}(\log(1+y))\bigr)\mathrm{d}y\\
			&\leq C\exp\bigl(\Omega_{\beta,T_0}(\log R)\bigr)
			R^{n/2+\beta}\int_{B(T_0)}^{R-1}\mathrm{d}y\leq CR^{1+n/2+\beta}
			\exp\bigl(\Omega_{\beta,T_0}(\log R)\bigr).
		\end{align*}
	\end{proof}
	
	\section{Direct fundamental inequality and Osgood argument} \label{sec:fundamental-osgood}
	
	Define the initial sign functional
	$$
	\mathcal{J}_0:=\rho(0)\int_{\mathbb{R}^n}g(x)\mathrm{d}x
	+\bigl(b(0)\rho(0)-\rho'(0)\bigr)
	\int_{\mathbb{R}^n}f(x)\mathrm{d}x.
	$$
	
	\begin{lemma}[Direct fundamental inequality]
		\label{lem:fundamental}
		Let $u$ be a weak solution on $[0,T)$ and suppose that $\mathcal{J}_0>0$. Put
		$$
		I_R:=\int_0^T\int_{\mathbb{R}^n} |u|^p\rho\psi_R\,\mathrm{d}x\mathrm{d}t,
		\quad
		I_R^*:=\int_0^T\int_{\mathbb{R}^n} |u|^p\rho\psi_R^*\,\mathrm{d}x\mathrm{d}t,
		$$
		where the integrals are over $[0,T)\times\mathbb{R}^n$. Then there exists $R_0\geq2$ such that
		\begin{equation}
			\label{eq:fundamental}
			\frac{\varepsilon\mathcal{J}_0}{2}+I_R
			\leq
			CR^{-\theta/p'}
			\exp\left(\frac{\Omega_{\beta,T_0}(\log R)}{p'}\right)
			(I_R^*)^{1/p},
			\quad
			\theta:=\frac{1}{p-1}-\frac{n}{2}-\beta,
		\end{equation}
		for every
		$$
		R_0\leq R<1+B(T),
		$$
		with the upper restriction omitted when $T=\infty$.
	\end{lemma}
	
	\begin{proof}
		For every admissible $R$, the function $\phi=\rho\psi_R$ belongs to $\mathcal{C}_c^2([0,T)\times\mathbb R^n)$ and is nonnegative. Since $\rho$ solves \eqref{eq:adjoint}, direct expansion gives
		$$
		\phi_{tt}-\Delta\phi-\partial_t(b\phi)+m^2\phi
		=\rho(\partial_t^2\psi_R-\Delta\psi_R)
		+(2\rho'-b\rho)\partial_t\psi_R.
		$$
		By \eqref{eq:rho-der}, $|2\rho'-b\rho|\leq Cb\rho$. Lemma \ref{lem:cutoff} therefore gives
		\begin{equation}
			\label{eq:Lstar-bound}
			\left|
			\phi_{tt}-\Delta\phi-\partial_t(b\phi)+m^2\phi
			\right|
			\leq\frac{C}{R}\rho(\psi_R^*)^{1/p}.
		\end{equation}
		
		The boundary term in \eqref{eq:weak} equals
		$$
		\mathcal B_R:=\varepsilon\int_{\mathbb R^n}
		\bigl[\rho(0)g\psi_R(0)
		+(b(0)\rho(0)-\rho'(0))f\psi_R(0)-\rho(0)f\partial_t\psi_R(0)\bigr] \,\mathrm{d}x.
		$$
		For every fixed $x$, $\psi_R(0,x)\to1$ as $R\to\infty$. Moreover, $0\leq\psi_R(0,x)\leq1$ and
		$$
		|\partial_t\psi_R(0,x)|\leq\frac{C}{Rb(0)}\leq C.
		$$
		The first two integrands are dominated by constants times $|f|+|g|$, while the last one converges pointwise to zero and is dominated by $C|f|$. Hence dominated convergence gives
		$$
		\mathcal B_R\longrightarrow\varepsilon\mathcal J_0.
		$$
		Since $\mathcal J_0>0$, there exists $R_0\geq2$, depending on the fixed data and coefficients but not on $\varepsilon$, such that
		$$
		\mathcal B_R\geq\frac{\varepsilon\mathcal J_0}{2}
		\quad(R\geq R_0).
		$$
		
		Using \eqref{eq:weak}, taking the absolute value of the right-hand side, and applying \eqref{eq:Lstar-bound}, we obtain
		$$
		\frac{\varepsilon\mathcal J_0}{2}+I_R
		\leq\frac{C}{R}
		\int_0^T\int_{\mathbb{R}^n} |u|\rho(\psi_R^*)^{1/p}\,\mathrm{d}x\mathrm{d}t.
		$$
		All integrals are finite because the cutoffs have compact support and $u\in L^p_{\mathrm{loc}}$. H\"older's inequality with respect to the measure $\rho(t)\,\mathrm{d}x\mathrm{d}t$ gives
		$$
		\frac{\varepsilon\mathcal J_0}{2}+I_R \leq\frac{C}{R}(I_R^*)^{1/p}\left(\int_0^T
		\int_{\left\{x\in\mathbb{R}^n:
			1+|x|^2+B(t)\leq R\right\}}\rho(t)\,\mathrm{d}x\mathrm{d}t\right)^{1/p'}.
		$$
		For fixed $t$, the spatial section is the ball of radius $[R-1-B(t)]_+^{1/2}$ and therefore has volume
		$$
		|B_1(0)|\bigl[R-1-B(t)\bigr]_+^{n/2}.
		$$
		Lemma \ref{lem:volume} yields
		$$
		\frac{\varepsilon\mathcal J_0}{2}+I_R
		\leq CR^{-1+(1+n/2+\beta)/p'}
		\exp\left(\frac{\Omega_{\beta,T_0}(\log R)}{p'}\right)
		(I_R^*)^{1/p}.
		$$
		Finally,
		$$
		-1+\frac{1+n/2+\beta}{p'}
		=\frac{-p'+1+n/2+\beta}{p'}=-\frac{1}{p'}
		\left(\frac{1}{p-1}-\frac n2-\beta\right)
		=-\frac{\theta}{p'},
		$$
		which proves \eqref{eq:fundamental}.
	\end{proof}
	
	\begin{lemma}[Osgood test-function lemma]
		\label{lem:osgood}
		Let $T\in(0,\infty]$ and let
		$$
		w\in L^1_{\mathrm{loc}}
		\bigl([0,T)\times\mathbb{R}^n\bigr), \quad w\geq0.
		$$
		Suppose that, for $R_0\leq R<\mathcal{R}$,
		\begin{align}
			\label{eq:osgood-ass}
			&\delta+\int_0^T\int_{\mathbb{R}^n}w(t,x)\psi_R(t,x)\,\mathrm{d}x\mathrm{d}t \notag\\
			\quad&\leq C_0R^{-\theta/p'}
			\exp\left(\frac{\Omega_{\beta,T_0}(\log R)}{p'}\right)\left(\int_0^T\int_{\mathbb{R}^n} w(t,x)\psi_R^*(t,x)\,\mathrm{d}x\mathrm{d}t
			\right)^{1/p},
		\end{align}
		where $\delta>0$ and $\theta\geq0$. Then
		\begin{equation}
			\label{eq:osgood-concl}
			\int_{R_0}^{\mathcal{R}}
			r^{-1+(p-1)\theta}
			\exp\bigl(
			-(p-1)\Omega_{\beta,T_0}(\log r)
			\bigr)\,\mathrm{d}r
			\leq
			C\delta^{-(p-1)}.
		\end{equation}
	\end{lemma}
	
	\begin{proof}
		Set
		$$
		y(r):=\int_0^T\int_{\mathbb{R}^n}w\psi_r^*,
		\quad
		Y(R):=\int_{R_0}^R\frac{y(r)}{r}\,\mathrm{d}r,
		$$
		where the integrals may initially take the value $+\infty$. Fix a spacetime point and write $S:=1+|x|^2+B(t)$. Since $\eta^*(S/r)$ can be nonzero only when $1/2\leq S/r\leq1$, one must have $S\leq r\leq2S$. If $S\leq R/2$, then $S/R\leq1/2$, so $\psi_R=1$, and
		$$
		\int_{R_0}^R\psi_r^*\frac{\mathrm{d}r}{r}
		\leq\int_S^{2S}\frac{\mathrm{d}r}{r}
		=\log2=(\log2)\psi_R.
		$$
		If $R/2<S\leq R$, then $S\leq r\leq R$ on the integration domain and $S/r\geq S/R$. Since $\eta$ is nonincreasing,
		$$
		\psi_r^*=\eta(S/r)^{2p'}\leq\eta(S/R)^{2p'}=\psi_R,
		$$
		so
		$$
		\int_{R_0}^R\psi_r^*\frac{\mathrm{d}r}{r}
		\leq\psi_R\int_S^R\frac{\mathrm{d}r}{r}
		\leq(\log2)\psi_R.
		$$
		If $S>R$, both sides vanish. Thus
		$$
		\int_{R_0}^R\psi_r^*(t,x)\frac{\mathrm{d}r}{r}
		\leq(\log2)\psi_R(t,x)
		$$
		for every point. Since $w\geq0$, Tonelli's theorem gives
		$$
		Y(R)=\int_0^T\int_{\mathbb{R}^n} w(t,x) \left(\int_{R_0}^R\psi_r^*(t,x)\frac{\mathrm{d}r}{r}\right)
		\mathrm d x\mathrm d t \leq(\log2) \int_0^T\int_{\mathbb{R}^n} w\psi_R<\infty.
		$$
		For every compact interval $[R_0,R_1]\subset[R_0,\mathcal R)$, the same argument shows that $y(r)/r\in L^1(R_0,R_1)$. Hence $Y$ is locally absolutely continuous and
		$$
		RY'(R)=y(R)
		$$
		for almost every $R$. Therefore, with $c=(\log2)^{-1}$, assumption \eqref{eq:osgood-ass} implies
		$$
		\delta+cY(R)
		\leq C_0R^{-\theta/p'}
		\exp\left(\frac{\Omega_{\beta,T_0}(\log R)}{p'}\right)
		y(R)^{1/p}.
		$$
		Raising this inequality to the power $p$ and using $y(R)=RY'(R)$ yields
		$$
		(\delta+cY(R))^p
		\leq C R^{1-(p-1)\theta}
		\exp\bigl((p-1)\Omega_{\beta,T_0}(\log R)\bigr)Y'(R).
		$$
		Consequently, for almost every $R$,
		$$
		\frac{Y'(R)}{(\delta+cY(R))^p}
		\geq c_1R^{-1+(p-1)\theta}
		\exp\bigl(-(p-1)\Omega_{\beta,T_0}(\log R)\bigr).
		$$
		Integrating from $R_0$ to any $R<\mathcal R$ and using $Y(R_0)=0$, we get
		\begin{align*}
			\int_{R_0}^R
			r^{-1+(p-1)\theta}
			\exp\bigl(-(p-1)\Omega_{\beta,T_0}(\log r)\bigr)\,\mathrm{d}r &\leq C\int_{R_0}^R\frac{Y'(r)}{(\delta+cY(r))^p}\,\mathrm{d}r\\
			&=\frac{C}{c(p-1)}
			\left[\delta^{1-p}-(\delta+cY(R))^{1-p}\right]
			\leq C\delta^{-(p-1)}.
		\end{align*}
		Letting $R\uparrow\mathcal R$ proves \eqref{eq:osgood-concl}; if $\mathcal R=\infty$, use monotone convergence.
	\end{proof}
	
	\section{Main nonexistence and lifespan results} \label{sec:main-results}
	
	\begin{theorem}[Direct $L^1$ nonexistence criterion]
		\label{thm:abstract-main}
		Assume Hypotheses \ref{ass:b}--\ref{ass:liouville}. Let $\rho$ be the slow adjoint solution from Theorem \ref{thm:rho}, and let $f,g\in L^1(\mathbb{R}^n)$ satisfy $\mathcal{J}_0>0$. Let $u$ be a weak solution of \eqref{eq:main} on $[0,T)$. Set
		$$
		\theta:=\frac{1}{p-1}-\frac{n}{2}-\beta.
		$$
		If $\theta\geq0$ and
		\begin{equation}
			\label{eq:divergence}
			\int_2^\infty
			R^{-1+(p-1)\theta}
			\exp\bigl(-(p-1)\Omega_{\beta,T_0}(\log R)\bigr)\mathrm{d}R
			=\infty,
		\end{equation}
		then no global weak solution exists. More precisely, if $T<\infty$, then either $1+B(T)\leq R_0$, or
		\begin{equation}
			\label{eq:life-implicit}
			\int_{R_0}^{1+B(T)}
			R^{-1+(p-1)\theta}
			\exp\bigl(-(p-1)\Omega_{\beta,T_0}(\log R)\bigr)\mathrm{d}R
			\leq C(\varepsilon\mathcal{J}_0)^{-(p-1)}.
		\end{equation}
	\end{theorem}
	
	\begin{proof}
		Apply Lemma \ref{lem:osgood} to \eqref{eq:fundamental} with
		$$
		w=|u|^p\rho,
		\quad
		\delta=\frac{\varepsilon\mathcal J_0}{2}.
		$$
		
		Suppose first that $T=\infty$. For every
		$\mathcal R>R_0$, Lemma \ref{lem:osgood} gives
		$$
		\int_{R_0}^{\mathcal R}
		R^{-1+(p-1)\theta}
		\exp\bigl(
		-(p-1)\Omega_{\beta,T_0}(\log R)
		\bigr)\,\mathrm{d}R
		\leq
		C(\varepsilon\mathcal J_0)^{-(p-1)}.
		$$
		Letting $\mathcal R\to\infty$ and applying the monotone convergence
		theorem, we obtain
		$$
		\int_{R_0}^{\infty}
		R^{-1+(p-1)\theta}
		\exp\bigl(
		-(p-1)\Omega_{\beta,T_0}(\log R)
		\bigr)\,\mathrm{d}R
		<\infty.
		$$
		Since $R_0\geq2$ and $\Omega_{\beta,T_0}$ is finite on bounded
		intervals, the integral over $[2,R_0]$ is finite. Hence the last
		conclusion contradicts \eqref{eq:divergence}. Therefore no global weak
		solution exists.
		
		Let now $T<\infty$. If $1+B(T)\leq R_0$, the first alternative holds.
		If $1+B(T)>R_0$, apply Lemma \ref{lem:osgood} with
		$$
		\mathcal R=1+B(T).
		$$
		It follows that
		$$
		\int_{R_0}^{1+B(T)}
		R^{-1+(p-1)\theta}
		\exp\bigl(
		-(p-1)\Omega_{\beta,T_0}(\log R)
		\bigr)\,\mathrm{d}R
		\leq
		C(\varepsilon\mathcal J_0)^{-(p-1)},
		$$
		which proves \eqref{eq:life-implicit}.
	\end{proof}
	
	\begin{corollary}[Strictly subcritical and critical powers]
		\label{cor:ranges}
		Under the assumptions of Theorem \ref{thm:abstract-main}, the following statements hold.
		\begin{enumerate}
			\item If $\Omega_{\beta,T_0}$ is sublinear, then no global weak solution exists for
			$$
			1<p<p_{\beta,1}(n)=1+\frac{2}{n+2\beta}.
			$$
			\item At $p=p_{\beta,1}(n)$, no global weak solution exists provided
			\begin{equation}
				\label{eq:critical-osgood}
				\int_0^\infty
				\exp\bigl(-(p_{\beta,1}(n)-1)\Omega_{\beta,T_0}(L)\bigr)\mathrm{d}L
				=\infty.
			\end{equation}
			\item Suppose, in addition, that for each sufficiently small $\varepsilon>0$ there exists a maximal solution $u_\varepsilon$ in a fixed local well-posedness class on $[0,T_\varepsilon)$ and that, for every $T<T_\varepsilon$, the restriction of $u_\varepsilon$ to $[0,T)$ is a weak solution in the sense of Definition \ref{def:weak}. If $\Omega_{\beta,T_0}$ is bounded, then
			\begin{equation}
				\label{eq:lifespan}
				1+B(T_\varepsilon)
				\leq
				\begin{cases}
					C\left(1+\varepsilon^{-\left(\frac{1}{p-1}-\frac{n}{2}-\beta\right)^{-1}}\right),
					&1<p<p_{\beta,1}(n),\\[1.2ex]
					C\exp\bigl(C\varepsilon^{-(p-1)}\bigr),
					&p=p_{\beta,1}(n).
				\end{cases}
			\end{equation}
			Here $C>0$ may depend on the coefficients, $p$, and the fixed data, but
			is independent of $\varepsilon$.
		\end{enumerate}
	\end{corollary}
	
	\begin{proof}
		In the strictly subcritical range, $\theta>0$. By sublinearity, for
		$$
		\sigma:=\frac{\theta}{2}>0
		$$
		there exists $L_\sigma$ such that $\Omega_{\beta,T_0}(L)\leq\sigma L$ for $L\geq L_\sigma$. Hence, for sufficiently large $R$,
		$$
		R^{-1+(p-1)\theta}
		\exp\bigl(-(p-1)\Omega_{\beta,T_0}(\log R)\bigr)
		\geq R^{-1+(p-1)\theta/2}.
		$$
		Since $(p-1)\theta/2>0$, the last power is not integrable at infinity, and \eqref{eq:divergence} follows.
		
		At $p=p_{\beta,1}(n)$, one has $\theta=0$. The substitution $L=\log R$ transforms the integral in \eqref{eq:divergence} into
		$$
		\int_{\log2}^\infty
		\exp\bigl(-(p-1)\Omega_{\beta,T_0}(L)\bigr)\,\mathrm{d}L.
		$$
		Because $\Omega_{\beta,T_0}$ is finite on bounded intervals, the integral over $[0,\log2]$ is finite and strictly positive. Thus divergence over $[\log2,\infty)$ is equivalent to \eqref{eq:critical-osgood}.
		
		Assume now that $\Omega_{\beta,T_0}$ is bounded, say $\Omega_{\beta,T_0}\leq M$. Theorem \ref{thm:abstract-main} first excludes $T_\varepsilon=\infty$. If $1+B(T_\varepsilon)\leq R_0$, both estimates in \eqref{eq:lifespan} are immediate after increasing $C$. Otherwise, for all $T<T_\varepsilon$ sufficiently close to $T_\varepsilon$, one has $1+B(T)>R_0$; apply \eqref{eq:life-implicit} and let $T\uparrow T_\varepsilon$. If $\theta>0$, then
		$$
		\frac{e^{-(p-1)M}}{(p-1)\theta}
		\left[(1+B(T_\varepsilon))^{(p-1)\theta}-R_0^{(p-1)\theta}\right]
		\leq C\varepsilon^{-(p-1)},
		$$
		where the positive fixed factor $\mathcal J_0^{-(p-1)}$ has been absorbed into $C$. Therefore
		$$
		1+B(T_\varepsilon)
		\leq C\left(1+\varepsilon^{-1/\theta}\right).
		$$
		If $\theta=0$, then
		$$
		e^{-(p-1)M}\log\frac{1+B(T_\varepsilon)}{R_0}
		\leq C\varepsilon^{-(p-1)},
		$$
		which gives
		$$
		1+B(T_\varepsilon)
		\leq C\exp\bigl(C\varepsilon^{-(p-1)}\bigr).
		$$
		These are exactly the bounds in \eqref{eq:lifespan}.
	\end{proof}
	
	\section{Coefficient examples satisfying both sets of assumptions} \label{sec:examples}
	
	The purpose of this section is to exhibit explicit pairs of damping and mass coefficients for which the global-in-time small-data existence result in \cite[Theorem 3]{DGR2019} applies in the supercritical range, while the nonexistence result in Corollary \ref{cor:ranges} of the present paper applies in the strictly subcritical and critical ranges.
	
	\begin{example}
		Let
		$$
		b(t)=\mu(1+t)^k,
		\quad
		m^2(t)=\mu(a+k)(1+t)^{k-1},
		$$
		where
		$$
		0 \leq k<1,
		\quad
		a>0,
		\quad
		\mu\geq4a+2k.
		$$
		Equivalently,
		$$
		m(t)=\sqrt{\mu(a+k)}(1+t)^{(k-1)/2}.
		$$
		This is a threshold polynomial family of \cite[Example 1]{DGR2019}.
		Indeed,
		$$
		\frac{b'(t)}{b(t)}=\frac{k}{1+t},
		\quad
		\frac{b''(t)}{b(t)}=\frac{k(k-1)}{(1+t)^2}, \quad
		tb(t)\longrightarrow\infty,
		\quad
		\frac{1}{b(t)(1+t)^2}
		=\frac{1}{\mu}(1+t)^{-k-2}\in L^1(0,\infty),
		$$
		and
		$$
		\frac{tb'(t)}{b(t)}=\frac{kt}{1+t}\leq k<1.
		$$
		Moreover,
		$$
		\frac{m'(t)}{m(t)}=\frac{k-1}{2(1+t)},
		\quad
		\frac{m''(t)}{m(t)}
		=\frac{(k-1)(k-3)}{4(1+t)^2},
		$$
		and
		$$
		\frac{m(t)}{b(t)}
		=\sqrt{\frac{a+k}{\mu}}(1+t)^{-(k+1)/2}
		\longrightarrow0.
		$$
		Hence Hypotheses 1--2 of \cite{DGR2019} are satisfied.
		
		Furthermore,
		$$
		B(t)=\frac{(1+t)^{1-k}-1}{\mu(1-k)},
		\quad
		Q(t,r)=(a+k)\log\frac{1+t}{1+r},
		$$
		and
		$$
		\beta=\lim_{t\to\infty}B(t)m^2(t)
		=\frac{a+k}{1-k}.
		$$
		Define
		$$
		G(t):=\log(1+B(t))-(1-k)\log(1+t).
		$$
		Since
		$$
		\lim_{t\to\infty}G(t)=-\log\bigl(\mu(1-k)\bigr),
		$$
		the function $G$ is bounded on $[0,\infty)$. Using
		$a+k=\beta(1-k)$, we obtain
		$$
		Q(t,r)-\beta\log X(t,r)
		=-\beta\bigl(G(t)-G(r)\bigr).
		$$
		Consequently,
		$$
		\Omega_{\beta,0}(L)
		\leq2\beta\|G\|_{L^\infty(0,\infty)}
		\quad(L\geq0),
		$$
		so Hypothesis \ref{ass:balance} holds with bounded intrinsic excess.
		Finally, writing $x=1+t\geq1$, we have
		$$
		\frac{b^2(t)}4+\frac{b'(t)}2-m^2(t)=\mu x^{k-1}
		\left(\frac{\mu}{4}x^{k+1}-a-\frac{k}{2}\right)\geq\mu x^{k-1}
		\left(\frac{\mu}{4}-a-\frac{k}{2}\right)\geq0.
		$$
		Thus Hypothesis \ref{ass:liouville} is also satisfied.
	\end{example}
	
	\begin{example}
		Let
		$$
		b(t)=\mu(1+t)^{-\kappa},
		\quad
		m^2(t)=\mu c(1+t)^{-\kappa-1},
		$$
		where
		$$
		0 \leq \kappa<1,
		\quad
		c>0,
		\quad
		\mu\geq4c+2\kappa.
		$$
		Then
		$$
		\frac{b'(t)}{b(t)}=-\frac{\kappa}{1+t}, \quad \frac{b''(t)}{b(t)}=\frac{\kappa(\kappa+1)}{(1+t)^2}, \quad
		tb(t)\longrightarrow\infty, \quad
		\frac{1}{b(t)(1+t)^2}
		=\frac{1}{\mu}(1+t)^{\kappa-2}\in L^1(0,\infty),
		$$
		and $tb'(t)\leq0$, so the growth condition holds with $a_0=0$.
		Moreover,
		$$
		m(t)=\sqrt{\mu c}(1+t)^{-(\kappa+1)/2}, \quad
		\frac{m'(t)}{m(t)}=-\frac{\kappa+1}{2(1+t)},
		\quad
		\frac{m''(t)}{m(t)}
		=\frac{(\kappa+1)(\kappa+3)}{4(1+t)^2},
		$$
		and
		$$
		\frac{m(t)}{b(t)}
		=\sqrt{\frac{c}{\mu}}(1+t)^{(\kappa-1)/2}
		\longrightarrow0.
		$$
		Thus Hypotheses 1--2 of \cite{DGR2019} are satisfied.
		
		Furthermore,
		$$
		B(t)=\frac{(1+t)^{\kappa+1}-1}{\mu(\kappa+1)},
		\quad
		Q(t,r)=c\log\frac{1+t}{1+r},
		$$
		and
		$$
		\beta=\lim_{t\to\infty}B(t)m^2(t)
		=\frac{c}{\kappa+1}.
		$$
		Define
		$$
		G_-(t):=\log(1+B(t))-(\kappa+1)\log(1+t).
		$$
		Since
		$$
		\lim_{t\to\infty}G_-(t)
		=-\log\bigl(\mu(\kappa+1)\bigr),
		$$
		the function $G_-$ is bounded. Since $c=\beta(\kappa+1)$,
		$$
		Q(t,r)-\beta\log X(t,r)
		=-\beta\bigl(G_-(t)-G_-(r)\bigr).
		$$
		Hence
		$$
		\Omega_{\beta,0}(L)
		\leq2\beta\|G_-\|_{L^\infty(0,\infty)},
		\quad L\geq0.
		$$
		Finally, with $x=1+t\geq1$,
		$$
		\frac{b^2(t)}4+\frac{b'(t)}2-m^2(t)=\mu x^{-\kappa-1}
		\left(\frac{\mu}{4}x^{1-\kappa}-\frac{\kappa}{2}-c\right)\geq\mu x^{-\kappa-1}
		\left(\frac{\mu}{4}-\frac{\kappa}{2}-c\right)\geq0.
		$$
		Thus all the required assumptions hold.
	\end{example}
	
	\begin{example}
		Let
		$$
		b(t)=\mu\frac{1+t}{a+t},
		\quad
		B(t)=\frac{t+(a-1)\log(1+t)}{\mu},
		\quad
		m^2(t)=\frac{\gamma}{1+B(t)},
		$$
		where
		$$
		a>1,
		\quad
		\mu>0,
		\quad
		0<\gamma\leq\frac{\mu^2}{4a^2}.
		$$
		The displayed formula for $B$ follows from direct integration of
		$b^{-1}$. Moreover,
		$$
		\frac{b'(t)}{b(t)}
		=\frac{a-1}{(1+t)(a+t)},
		\quad
		\frac{b''(t)}{b(t)}
		=-\frac{2(a-1)}{(1+t)(a+t)^2}.
		$$
		Thus the derivative estimates hold, $b$ is increasing, and
		$$
		\sup_{t\geq0}\frac{tb'(t)}{b(t)}
		=\frac{\sqrt a-1}{\sqrt a+1}<1.
		$$
		Also,
		$$
		tb(t)\longrightarrow\infty,
		\quad
		\frac{1}{b(t)(1+t)^2}
		=\frac{a+t}{\mu(1+t)^3}\in L^1(0,\infty).
		$$
		
		Set
		$$
		h(t):=\frac{1}{b(t)(1+B(t))}.
		$$
		Since $b(t)\geq\mu/a$ and $B(t)\geq t/\mu$, we have
		$$
		0<h(t)\leq\frac{C}{1+t}.
		$$
		Furthermore,
		$$
		h'(t)=-h(t)\left(\frac{b'(t)}{b(t)}+h(t)\right),
		$$
		so $|h'(t)|\leq C(1+t)^{-2}$. Since
		$$
		m(t)=\sqrt\gamma\,(1+B(t))^{-1/2},
		\quad
		\frac{m'(t)}{m(t)}=-\frac{h(t)}2,
		$$
		we obtain
		$$
		\frac{m''(t)}{m(t)}
		=-\frac{h'(t)}2+\frac{h^2(t)}4,
		$$
		and therefore the derivative estimates for $m$ hold. Moreover,
		$$
		\frac{m(t)}{b(t)}\longrightarrow0
		$$
		because $B(t)\to\infty$ and $b(t)\geq\mu/a$.
		
		By construction,
		$$
		Q(t,r)
		=\int_r^t\frac{\gamma}{b(\tau)(1+B(\tau))}\,\mathrm d\tau
		=\gamma\log X(t,r),
		$$
		so
		$$
		\Omega_{\gamma,0}(L)=0
		\quad(L\geq0).
		$$
		Also,
		$$
		\lim_{t\to\infty}B(t)m^2(t)
		=\gamma\lim_{t\to\infty}\frac{B(t)}{1+B(t)}=\gamma.
		$$
		Finally, $b'(t)\geq0$, $m^2(t)\leq\gamma$, and
		$b(t)\geq\mu/a$. Hence
		$$
		m^2(t)\leq\gamma
		\leq\frac{\mu^2}{4a^2}
		\leq\frac{b^2(t)}4
		\leq\frac{b^2(t)}4+\frac{b'(t)}2.
		$$
		Thus all coefficient assumptions of both frameworks are satisfied.
	\end{example}
	
	\section{Concluding remarks} \label{sec:conclusion}
	
	The main structural improvement is that the positive adjoint multiplier is no longer postulated. It is constructed from the effective-damping, dominated-mass, accumulated-balance, and Liouville nonoscillation conditions.
	
	The proof separates the large-time and finite-time mechanisms. At large time, the canonical slow mode is obtained from a terminal Volterra equation and has size $b(t)^{-1}e^{Q(t,T_*)}$. On the compact initial interval, the Liouville condition
	$$
	m^2(t)\leq\frac{b^2(t)}4+\frac{b'(t)}2
	$$
	prevents the transformed slow mode from acquiring a zero. Since this inequality is automatic at sufficiently large times under the original asymptotic assumptions, it is essentially a compact-time nonoscillation requirement.
	
	The intrinsic excess $\Omega_{\beta,T_0}$ has a different role: it controls the accumulated growth of the slow multiplier and determines the nonexistence scale. Under bounded excess, the method gives the conditional polynomial and exponential upper lifespan bounds stated in \eqref{eq:lifespan}. The coefficient families in Section 7 satisfy simultaneously Hypotheses 1--2 of \cite{DGR2019} and all coefficient assumptions required by the present nonexistence theorem.
	
	\section*{Acknowledgments}
	Duc An Phan sincerely acknowledges the financial support provided by the Banking Academy of Vietnam. 
	

\end{document}